\documentclass[11pt] {article}
\usepackage{amsmath,amssymb,amsfonts}
\usepackage[utf8]{inputenc} 
\usepackage{amsthm}

\usepackage{tikz}
\usepackage{hyperref}
\usepackage{graphicx}
\usepackage{amsfonts}
\usepackage{xcolor}

\usepackage{color}
\usepackage{ifthen}
\usepackage{graphicx}
\usepackage{tikz}
\usetikzlibrary{decorations.markings}
\usetikzlibrary{plotmarks}
\usetikzlibrary{patterns}
\usepackage{pgfplots}
\usepgfplotslibrary{fillbetween}
\usepackage{makeidx}
\usetikzlibrary{calc}

\usepackage{geometry} 
\usepackage{graphicx} 

\usepackage{array} 
\usepackage{paralist} 
\usepackage{verbatim} 
\usepackage{subfig} 

\usepackage{textpos}

\newcommand{\R}{\mathbb{R}}

\newcommand{\D}{\mathbb{D}}

\newcommand{\n}{\vec{n}}
\newcommand{\bY}{\vec{Y}}
\newcommand{\s}{\mathbb{S}}
\newcommand{\Ar}{\mathring{A}}
\newcommand{\Er}{\mathring{E}}
\newcommand{\bp}{{\Phi}}
\newcommand{\bpsi}{{\Psi}}

\tikzset { domaine/.style 2 args={domain=#1:#2} }

\newtheorem{theo}{Theorem}[section]
\newtheorem*{theo*}{Theorem}

\newtheorem{prop}[theo]{Proposition}
\newtheorem*{prop*}{Proposition}
\newtheorem{lem}{Lemma}[section]
\newtheorem{cor}{Corollary}[section]
\newtheorem*{cor*}{Corollary}
\newtheorem{de}{Definition }[section]

\newtheorem{remark}{Remark}[section]

\newcommand{\nocontentsline}[3]{}
\newcommand{\tocless}[2]{\bgroup\let\addcontentsline=\nocontentsline#1{#2}\egroup}

\title{Energy Estimates for the Tracefree Curvature\\ of Willmore Surfaces and Applications}
\author{Yann Bernard\footnote{School of Mathematics, Monash University, 3800 Clayton, Victoria, Australia.}\:\:,\:Paul Laurain\footnote{Institut Mathématique de Jussieu, Université de Paris, Bâtiment Sophie Germain, Case 7052, 75205 Paris Cédex 13, France \& DMA, Ecole normale supérieure, CNRS, PSL Research University, 75005 Paris}\:\:,\:Nicolas Marque\footnote{University of Potsdam, Institute for Mathematics,
Karl-Liebknecht-Straße 24/25, 14476 Potsdam, Germany}}
\begin{document}
\maketitle
\abstract{We prove an $\varepsilon$-regularity result for the tracefree curvature of a Willmore surface with bounded second fundamental form. For such a surface, we obtain a pointwise control of the tracefree second fundamental form from a small control of its $L^2$-norm. Several applications are investigated. Notably, we derive a gap statement for surfaces of the aforementioned type. We further apply our results to deduce regularity results for conformal minimal spacelike immersions into the de Sitter space $\mathbb{S}^{4,1}$.}
\tableofcontents
\section{Introduction}

\subsection{The Willmore Energy}

We consider an immersion $\bp$ from a closed Riemann surface $\Sigma$ into $\R^3$.  We denote by $g:= \bp^*\xi$ the induced metric on $\Sigma$, with $\xi$ the standard Euclidean metric on $\R^3$, and $d \mathrm{vol}_g$ the volume form associated with $g$. We denote by $\n$ the Gauss map of $\bp$, that is, the normal to the surface. In local coordinates $(x,y)$, we have $$\n := \frac{\bp_x \times \bp_y}{\left|\bp_x \times \bp_y \right|},$$
with $\bp_x = \partial_x \bp$, $\bp_y = \partial_y \bp$, and $\times$ is the usual cross product in $\R^3$.
The second fundamental form of $\bp$ is then defined as:
$$\vec{A} (X,Y) := A (X,Y) \n := \langle d^2 \bp \left( X, Y \right), \n \rangle \n.$$
The two key objects of this paper are the mean curvature $H$ and the trace-free second fundamental form $\Ar$ defined as follows:
$$ \vec{H}(p)= H(p) \n= \frac{1}{2} \text{Tr}_g \left( A \right) \n,$$
$$\vec{\Ar} (X,Y) = \Ar(X,Y)\vec{n} = \big(A(X,Y)  - H(p) g(X,Y)\big)\vec{n}.$$
From these, the Willmore energy is defined as $$W( \bp) := \int_\Sigma H^2 d \mathrm{vol}_g. $$
The Willmore energy was introduced in the early XIX$^\text{th}$ century to study elastic plates. It was identified as a conformal invariant by W. Blaschke (see \cite{MR0076373}) and further studied by T. Willmore (\cite{bibwill}). It must be pointed out that the conformal invariance of $W$ is \emph{contextual}: an inversion centered on a round sphere sends it to a plane with a loss of Willmore energy of $4 \pi$. The true \emph{pointwise} conformal invariant (see T. Willmore's \cite{bibwill}) is rather $| \Ar |_g^2 d \mathrm{vol}_{g}$. The tracefree total curvature is then  a conformal invariant, defined as:
$$\Er(\Phi) := \int_\Sigma | \Ar |_{g}^{2} d \mathrm{vol}_g.$$
Straightforward computations (see, for instance, appendix A in \cite{bibnm3}) show that,
with $\chi(\Sigma)$ denoting the Euler characteristic of $\Sigma$, one has:
\begin{equation} \label{courbsanstrace} \Er(\Phi) = 2W(\Phi) - 4\pi \chi(\Sigma).\end{equation}
The contextual conformal invariance of $W$ is thus to be understood as follows: $W$ is invariant under the action of conformal transformations that do not change the topology of the surface.\\

In the present article we will study Willmore immersions, that is critical points of $W$ (or equivalently $\Er$, given in \eqref{courbsanstrace}).  Willmore immersions form a conformally invariant family satisfying the Willmore equation:
\begin{equation}\label{willmoche}
\Delta_g H + |\Ar|_g^2H =0.
\end{equation}
Given the lackluster analytic properties of this equation (which is supercritical in the weak framework we will make explicit below), two pivotal results are the small energy estimates (termed $\varepsilon$\emph{-regularities}). The first of those is an \emph{extrinsic} result by E. Kuwert and R. Schätzle (see theorem 2.10 in  \cite{bibkuwschat}), followed by an \emph{intrinsic} version by T. Rivière (theorem I.5 in \cite{bibanalysisaspects}). We refer the readers to the discussion in the introduction of Y. Bernard, G. Wheeler, and V. Wheeler's \cite{bernardwheeler} for an explanation of why these two results do not overlap and fundamentally differ in philosophy, while a concrete counter-example can be found in \cite{bibnm3}, with further explanations in remark 3.3.3 of \cite{nmthese}.\\

The goal of the present paper is to prove an $\varepsilon$-regularity result  for the  trace-free second fundamental form, using T. Rivière's formalism of weak Willmore immersions. Such a search is motivated first by a similar \emph{extrinsic} result obtained with E. Kuwert and R. Schätzle's approach (see theorem 2.9 in \cite{bibkuwschat}), but also by a recent intrinsic $\varepsilon$-regularity result for the mean curvature  (theorem 1.4 in \cite{bibnmheps}). An $\Ar$ $\varepsilon$-regularity result would thus complete the extent of possibilities offered by the weak Willmore immersions formalism. Further, the application of the $H$ $\varepsilon$-regularity to minimal bubbling (by eliminating several minimal  bubbling configurations, see \cite{bibnm3}) unlocks the possibility of applying the $\Ar$ version to control the appearance of round spheres as Willmore bubbles (more details on Willmore bubbling can be found in \cite{bibenergyquant} and \cite{MR3843372}).

\subsection{Conformal Weak Willmore Immersions}

In this subsection we establish the notation we will adopt and the notions we will use throughout this paper. \\

Let $\Sigma$ be an arbitrary closed compact two-dimensional manifold and $g_0$ be a smooth ``reference" metric on $\Sigma$. The Sobolev space $W^{k,p} \left( \Sigma, \R^3 \right)$ of measurable maps from $\Sigma$ into $\R^3$ is defined as 
$$W^{k,p} \left( \Sigma , \R^3 \right) := \left\{ f \text{ measurable : } \Sigma \rightarrow \R^3 \text{ s.t } \sum_{l= 0}^{k} \int_\Sigma \left| \nabla_{g_0}^l f \right|^p_{g_0} d\mathrm{vol}_{g_0} < \infty \right\}.$$
Since $\Sigma$ is assumed to be compact, this definition does not depend on $g_0$.


\begin{de}
\label{weakimmersions}
Let $\bp:\Sigma \rightarrow \R^3$. Let $g_{\bp} = \bp^*\xi$ be the first fundamental form of $\bp$ and $\n$ its Gauss map.  Then $\bp$ is called a weak immersion with locally $L^2$-bounded second fundamental form if $\Phi \in W^{1,\infty} \left( \Sigma \right)$, if there exists a constant $C_\bp$ such that 
$$ \frac{1}{C_\bp} g_0  \le g_\bp  \le C_\bp g_0,$$
and if $$\int_\Sigma \left| d \n \right|^2_{g_\bp} d \mathrm{vol}_\bp < \infty.$$ The set of weak immersions with $L^2$-bounded second fundamental form on $\Sigma$ will be denoted $\mathcal{E}(\Sigma)$.
\end{de}

Weak immersions are regular enough for us to work with conformal charts, as seen in theorem 5.1.1 of \cite{bibharmmaps}.
\begin{theo}
\label{localconformalcoordinates}
Let $\bp \in \mathcal{E}(\Sigma)$ and $\D$ the unit disk in $\R^2$. Then for every $x \in \Sigma$, there exists an open disk $ D$ in $\Sigma$ containing $x$ and a homeomorphism $\Psi \, : \, \D\rightarrow D$ such that  $\bp \circ \Psi$ is a conformal bilipschitz immersion. The induced metric $g = \left(\bp \circ \Psi \right)^* \xi$
 is continuous. Moreover, the Gauss map $\n$ of this immersion is an element of $W^{1,2} \left( \D, \s^2 \right)$.
\end{theo}


Further exploiting the Green's function of $ \Sigma$, P. Laurain and T. Rivi\`ere have proven the existence of a specific atlas with higher regularity on the conformal factor (see theorem 3.1 of  \cite{biboptimalestimates}):
\begin{theo}
\label{theooncontrolelegraddufactconf}
Let $(\Sigma, g)$ be a closed Riemann surface of fixed genus greater than one. Let $h$ denote the metric with constant curvature (and volume equal to one in the torus case) in the conformal class of $g$ and $\bp \in \mathcal{E}(\Sigma)$ conformal, that is : 
$$\bp^* \xi = e^{2u} h.$$ 
Then, there exists a finite conformal atlas  $(U_i, \Psi_i)$ and a positive constant $C$ depending only on the genus of $\Sigma$, such that\footnote{The weak-$L^2$ Marcinkiewicz space $L^{2,\infty}(B_1(0))$ is defined as those functions $f$ which satisfy $\:\sup_{\alpha>0}\alpha^2\Big|\big\{x\in B_1(0)\,;\,|f(x)|\ge\alpha\big\}\Big|<\infty$. In dimension two, the prototype element of $L^{2,\infty}$ is $|x|^{-1}\,$. The space $L^{2,\infty}$ is also a Lorentz space, and in particular is a space of interpolation between Lebesgue spaces. See \cite{bibharmmaps} for details.} 
$$ \left\| \nabla \lambda_i \right\|_{L^{2, \infty} \left( V_i \right)} \le  C \left\| \nabla_{\bp^* \xi} \n \right\|^2_{L^2 \left( \Sigma \right)},$$
with $\lambda_i = \frac{1}{2} \log \frac{ \left| \nabla \bp\right|^2}{2}$ the conformal factor of $\bp \circ \Psi_i^{-1}$ in $V_i = \Psi_i (U_i)$.
\end{theo}

One can then automatically study any $ {\Phi} \in \mathcal{E} \left(\Sigma \right)$ in such  local conformal charts defined on the unit disk $\D$, as a conformal bilipschitz map ${\Phi} \in \mathcal{E}\left(\D\right)$. 
Nevertheless, if the conformal class degenerates when studying a sequence, the chart of the collar will be conformally equivalent to degenerating annuli. \\

For the sake of brevity, we set once and for all the notation pertaining to $\bp$ that we will adopt, namely

\begin{de}
Let ${\Phi} \in \mathcal{E} \left( \D \right)$ be a weak conformal immersion. We will denote:
\begin{itemize}
\item $\lambda$ its conformal factor, i.e. $e^{2\lambda} dxdy =\bp^*\xi$,
\item $\n$ its Gauss map,
\item $A:= \left\langle \nabla^2 \bp , \n \right\rangle$ its second fundamental form,
\item $H:= \frac{e^{-2\lambda}}{2} \mathrm{Tr}(A)$ its mean curvature with $\vec{H}:=H\vec{n}$,
\item $ \Ar:= A - H e^{2\lambda}\text{Id}$ its tracefree second fundamental form.
\end{itemize}
\end{de}
\noindent
That all these quantities are well defined while requiring as little regularity as possible on $\bp$ is a key reason to adopt the weak formalism to study Willmore immersions.
\begin{remark}
It must be pointed out that $ \big| \Ar \big|^2 d \mathrm{vol}_g  = \big| \Ar e^{- \lambda} \big|^2 dx dy$. The conformal invariant is then written in a local conformal chart: $ \big| \Ar e^{- \lambda} \big|$. It is this quantity which will appear in our estimates. 
\end{remark}

We now recall the notion of weak Willmore immersions (definition I.2 in \cite{bibanalysisaspects}):
\begin{de}

 Let $\bp\in \mathcal{E} \left( \Sigma \right)$. $\bp$  is a weak Willmore immersion if 
\begin{equation} \label{equationwillmorefaible} \mathrm{div} \left(  \nabla \vec{H} -3 \pi_{\n} \left( \nabla \vec{H} \right) + \nabla^\perp \n \times \vec{H} \right) = \vec{0} \end{equation}
holds in a distributional sense in every conformal parametrization $\Psi  \, : \, \D \rightarrow D$ on every neighborhood $D$ of $x$, for every $x \in \Sigma$. Here the operators $\mathrm{div}$, $\nabla$ and $\nabla^\perp = \begin{pmatrix} - \partial_y \\ \partial_x \end{pmatrix} $ are to be understood with respect to the flat metric on $\D$.
\end{de}
Equation \eqref{equationwillmorefaible} is simply the Willmore equation (\ref{willmoche}) written in divergence form well-defined for weak immersions. Thanks to the $\varepsilon$-regularity (theorem I.5 in \cite{bibanalysisaspects}) Willmore immersions are known to be smooth and thus to satisfy the  Willmore equation (\ref{willmoche}) in the classical sense.

\subsection{Conformal Gauss map}
Introduced by R. Bryant (\cite{bibdualitytheorem}, see also J.-H. Eschenburg's \cite{bibeschenburg}), the conformal Gauss map has proven a precious tool in the study of Willmore immersions, and will be pivotal in the present work. We thus briefly review its main properties. 

\begin{de}
Let $\bp\,: \, \Sigma  \rightarrow \R^3$ be an immersion. The conformal Gauss map of $\bp$ is a spacelike application $Y \, : \, \Sigma \rightarrow \s^{4,1} \subset \R^{4,1}$  defined as follows:
\begin{equation} 
\label{230520200727}
Y:= H \begin{pmatrix} \bp \\ \frac{|\bp|^2-1}{2} \\ \frac{|\bp|^2+1}{2} \end{pmatrix} + \begin{pmatrix} \n \\ \langle \n, \bp \rangle \\ \langle \n, \bp\rangle \end{pmatrix}.
\end{equation}
\end{de}
Throughout this paper, $\s^{4,1}$ denotes the de Sitter space in the Lorentz space $\R^{4,1}$, that is, the set of vectors $v$ satisfying $\langle v, v\rangle_{4,1}=1$ with $\langle . , . \rangle_{4,1}$ denoting the Lorentz product in $\R^{4,1}$, not to be confused with the Euclidean product $\langle . , . \rangle$.
Differentiating \eqref{230520200727}, one finds:
\begin{equation} \label{230520200910}
\nabla Y = \nabla H \begin{pmatrix} \bp \\ \frac{ | \bp|^2 -1}{2} \\ \frac{ |\bp|^2 + 1}{2} \end{pmatrix} -e^{-2\lambda} \Ar \begin{pmatrix}\nabla  \bp\\ \langle \nabla \bp, \bp \rangle \\ \langle \nabla \bp, \bp \rangle \end{pmatrix}.
\end{equation}

The conformal Gauss map is deeply linked with conformal geometry, as seen in the following proposition  (which is a merger of theorem 2.4 and proposition 3.3 in \cite{bibnmconfgaussmap}):

\begin{prop}
\label{modificationYtransformationconforme}
Let $\varphi \in \mathrm{Conf}( \R^3)$ and  $\bp \, : \, \Sigma \rightarrow \R^3$ be a smooth  immersion with conformal Gauss map $Y$.  We assume the  set of umbilic points of $\bp$ to be nowhere dense.
Let $Y_\varphi$ be the conformal Gauss map of $\varphi \circ Y$.
Then there exists $M \in SO(4,1)$  such that: $$Y_\varphi = M Y.$$
\end{prop}
\begin{remark}
The connection between $\varphi$ and $M$ is explicitly known, and we refer the reader to equalities (5)-(8) in \cite{bibnmconfgaussmap} for details. We will use the corresponding matrices of inversion and translation in \eqref{230520200815} below.
\end{remark}

The following proposition (see proposition 2 in J.-H. Eschenburg's \cite{bibeschenburg} or theorem 4.2 in \cite{bibnmconfgaussmap}) will clarify our interest in the conformal Gauss map in the present context:

\begin{prop}
\label{230520200838}
Let $\bp \,: \, \D \rightarrow \R^3$ be a conformal immersion. Its conformal Gauss map $Y \, : \, \D \rightarrow \s^{4,1}$ is a conformal map with conformal factor $| \Ar e^{-\lambda}|^2$.
Further, $Y$ is minimal if and only if $\bp$ is Willmore.
\end{prop}

\subsection{Summary of Main Results}

The main result of our paper is
\begin{theo}
\label{mainth1}
Let $\bp\in \mathcal{E}\left(\D \right)$ be a conformal weak Willmore immersion with normal vector $\vec{n}$ and conformal parameter $\lambda$.
Assume 
\begin{equation}
\label{210520201541}
\int_{\D} \left| \nabla \n \right|^2 \le \frac{4\pi}{3},
\end{equation}
and
\begin{equation}
\label{210520201621}
\left\| \nabla \lambda \right\|_{L^{2,\infty} \left( \D \right) } \le C_0,
\end{equation}
for some constant $C_0>0$. \\
There exist constants $\varepsilon_0(C_0)>0$ and $C(C_0) >0$ such that, if $$\left\|  \Ar e^{-\lambda} \right\|_{L^2\left( \D \right)} \le \varepsilon_0,$$ then
$$ \left\| \Ar  e^{-\lambda} \right\|_{L^\infty \left( \D_{\frac{1}{2}} \right) } \le C \left\|  \Ar e^{-\lambda} \right\|_{L^2 \left( \D \right)}.$$
\end{theo}

The small energy hypothesis \eqref{210520201541} also appears in theorem I.5 of \cite{bibanalysisaspects}, the difference here is that the resulting inequality only involves the tracefree curvature without relying on the whole second fundamental form. We will in fact establish more than theorem \ref{mainth1} and we show how to recover theorem I.5 of \cite{bibanalysisaspects} up to a conformal transformation whose conformal factor is controlled. 
This is the main idea of the proof: if $\Vert \mathring{A}_\bp\Vert_2$ is small enough we can find a conformal transformation $T$ such that $\Vert A_{\tilde{\bp}}\Vert_2$ is small, where $\tilde{\bp}=T\circ \bp$.

\begin{cor}
    \label{230520201048}
    Let $\bp \in \mathcal{E}\left(\D \right)$  be a conformal weak Willmore immersion.

    Assume 
    \begin{equation}
    \label{210520201541bus}
    \int_{\D} \left| \nabla \n \right|^2 \le \frac{4\pi}{3},
    \end{equation}
    and
    \begin{equation}
    \label{210520201621}
    \left\| \nabla \lambda \right\|_{L^{2,\infty} \left( \D \right) } \le C_0,
    \end{equation}
for some constant $C_0>0$. \\
    There exist constants $\varepsilon_0(C_0)>0$, $C(C_0) >0$ such that
    if $$\left\| \Ar e^{-\lambda} \right\|_{L^2\left( \D_\rho \right)} \le \varepsilon_0,$$
    then one can find a conformal transformation $\Theta$ such that, setting $\tilde{\bp}=\Theta\circ \bp$ and denoting $\vec{n}_{\tilde{\bp}}$ its Gauss map, one has
    $$ \frac{\Vert \nabla\bp\Vert_\infty}{C}\leq \Vert \nabla \tilde{\bp} \Vert_\infty \leq C  \Vert \nabla \bp \Vert_\infty$$ 
    and    
    $$ \left\| \nabla \vec{n}_{\tilde{\bp}} \right\|_{L^\infty \left( \D_{\frac{1}{2}} \right) } \le C \left\| \Ar e^{-\lambda} \right\|_{L^2 \left( \D \right)}.$$
\end{cor}

The next achievement of the present paper consists in removing the small energy hypothesis in theorem \ref{mainth1}, as indicated in the next statement. 

\begin{theo}
    \label{mainth2}
    Let $\bp \in \mathcal{E}\left(\D \right)$ be a conformal weak Willmore immersion.
    Assume 
    \begin{equation}
    \label{Mest1}
    \int_{\D} \left| \nabla \n \right|^2 \le C_0,
    \end{equation}
    and
    \begin{equation}
    \label{Mest2}
    \left\| \nabla \lambda \right\|_{L^{2,\infty} \left( \D \right) } \le C_0,
    \end{equation}
for some constant $C_0>0$. \\
    There exists constants $\varepsilon_0(C_0)>0$ and $C(C_0) >0$ such that, if $$\left\| \Ar e^{-\lambda} \right\|_{L^2\left( \D \right)} \le \varepsilon_0,$$ then
    \begin{equation} \label{Mest3} \left\| \Ar e^{-\lambda} \right\|_{L^\infty \left( \D_{\frac{1}{2}} \right) } \le C \left\| \Ar e^{-\lambda} \right\|_{L^2 \left( \D \right)}.\end{equation}
\end{theo}

By making the change of variables $\bp_\rho = \bp ( \rho\,\cdot) $ (as in \cite{bibnmheps}) one can easily extend these results to disks of arbitrary radius $\rho$:

\begin{cor}
\label{coro}
Let $\bp \in \mathcal{E}\left(\D_\rho \right)$  be a conformal weak Willmore immersion.
Assume 
\begin{equation}
\label{210520201541bus}
\int_{\D_\rho} \left| \nabla \n \right|^2 \le C_0,
\end{equation}
and
\begin{equation}
\label{210520201621}
\left\| \nabla \lambda \right\|_{L^{2,\infty} \left( \D_\rho \right) } \le C_0,
\end{equation}
for some constant $C_0>0$. \\
Then there exists a constant $\varepsilon_0(C_0)>0$  such that, if $$\left\| \Ar e^{-\lambda} \right\|_{L^2\left( \D_\rho \right)} \le \varepsilon_0,$$ there exists a constant $C(C_0) >0$ such that
$$ \left\| \Ar e^{-\lambda} \right\|_{L^\infty \left( \D_{\frac{\rho}{2}} \right) } \le \frac{C}{\rho} \left\| \Ar e^{-\lambda} \right\|_{L^2 \left( \D_\rho \right)}.$$
\end{cor}

It should be noted that the hypothesis \eqref{210520201621} is needed for analytical reasons as a minimal starting control on the metric. We have chosen  \eqref{210520201621} for convenience, but much weaker hypotheses are possible. We will explain why in practice  \eqref{210520201621} suffices. In addition, estimates involving the half-disk $\D_{\rho/2}$ can be replaced by any proper subdisk of $\D$, with appropriate adjustments on the constants involved.\\

We will also bring forth two applications of these theorems. First, we will develop a translation of the $\varepsilon$-regularity into the conformal Gauss map framework to obtain a result for minimal surfaces into the de Sitter space $\s^{4,1}$.

\begin{theo}
\label{theoconfGauss}
Let $Y$ be the conformal Gauss map of a conformal weak Willmore immersion satisfying \eqref{Mest1} and \eqref{Mest2}. There exists $\varepsilon_0>0$ such that if $$\langle \nabla Y , \nabla Y \rangle_{4,1} \le \varepsilon_0,$$ then:
\begin{equation} \label{epsregY} \left\| \langle \nabla Y, \nabla Y \rangle_{4,1} \right\|^2_{L^\infty \left(\D_{\frac{1}{2}} \right)} \le C \int_\D \langle \nabla Y , \nabla Y \rangle_{4,1}.\end{equation}
\end{theo}

It is known that any minimal spacelike immersion in $\s^{4,1}$ is the conformal Gauss map of a Willmore immersion (see for instance theorem 3.9 of \cite{bibnmconfgaussmap}). But it is also well-known that minimal surfaces, and more generally harmonic maps, play a key role in Physics, see \cite{Jost} or \cite{Deligne} for instance.
For example, they are at the core of the AdS/CFT correspondence proposed by Maldacena in 1998 \cite{Ma}, and more recently in his work with Alday \cite{AMa}. In General Relativity, space-time is represented by a Lorentzian manifold \cite{Oneill} whose de Sitter space $\mathbb{S}^{4,1}$ is one important cosmological model. In this context, the following reformulation of  Theorem \ref{theoconfGauss} should find important applications, at least since $\varepsilon$-regularity is the first step of any asymptotic analysis of sequences of harmonic maps.

\begin{cor}
\label{corconfGauss}
Let $Y\, : \, \Sigma \rightarrow \s^{4,1}$ be  a conformal minimal spacelike immersion such that 
$$ \int_\Sigma \langle \nabla Y , \nabla Y \rangle_{4,1} < \infty.$$  Then:
\begin{itemize}
\item
It is the conformal Gauss map of a Willmore immersion $\bp$ of finite total curvature.
\item
$Y$ satisfies \eqref{epsregY} around every point.
\end{itemize}
\end{cor}
\begin{proof}
Since $Y$ has finite energy, the Willmore immersion $\bp$ has finite tracefree total curvature, and thus, owing to the Gauss Bonnet formula, has finite total curvature.  Hypotheses  \eqref{Mest1} and \eqref{Mest2} (thanks to theorem \ref{theooncontrolelegraddufactconf}) are then satisfied on any local disk, and one can apply theorem \ref{theoconfGauss}.
\end{proof}

This result is to be compared with previous $\varepsilon$-regularity results for harmonic surfaces into $\s^{4,1}$, in particular theorem 1.7 in \cite{MR3003285} by M. Zhu:
\begin{theo}
\label{theoZhu}
Any weakly harmonic map $u \in W^{1,2}( \D, \s^{4,1})$ is H\"older continuous. In particular there exists $\varepsilon_0>0$ such that if $\| \nabla u \|_{L^2 (\D)} \le \varepsilon_0$ then
\begin{equation}
\label{epsregZhu}
\| \nabla u \|_{L^{\infty} \left(\D_{\frac{1}{2} }  \right)} \le C  \| \nabla u \|_{L^{2} \left(\D \right)}. \end{equation}
\end{theo}

The key difference between \eqref{epsregY} and \eqref{epsregZhu} lies in the norms they involve. In the case of \eqref{epsregY}, we make use of the more geometrically meaningful Lorentz norm, while M. Zhu's result involves the Euclidean norm, more convenient analytically but less convenient geometrically.
Corollary \ref{corconfGauss} thus improves on  theorem 1.7 of \cite{MR3003285}. The offset is that  the harmonic hypotheses are no longer sufficient, and we need to further assume minimality, as well as \eqref{Mest1} and \eqref{Mest2} (indeed, corollary \ref{corconfGauss} is more geometric in nature than analytic).\\

Our second application is a new gap result, which can be seen as a companion to the gap results of G. Wheeler and J. McCoy (see theorem 1 in \cite{WM}, or \cite{MR3314084}). While in the latter, it is assumed that the immersion is proper, we instead assume our immersion has bounded energy. 

\begin{theo} 
    \label{gap}
    There exists $\varepsilon_0 >0$ such any complete Willmore surface $\Sigma$ satisfying 
    $$\int_\Sigma \vert A\vert^2 \, d\text{vol}_g <+\infty,$$
 and   $$\int_\Sigma \vert \mathring{A}\vert^2 \, d\text{vol}_g  <\varepsilon_0,$$
is totally umbilic, i.e. is either a plane or a round sphere.
\end{theo}

\medskip

{\bf Acknowledgments: } Parts of this work were completed while the second and third authors were guests at Monash University under the {\it Robert Bartnik Fellowship} funding program. 
This work was also partially supported by the ANR BLADE-JC ANR- 18-CE40-002.


\section{Proof of the Main Theorems}

The proof will proceed in four broad steps:
\begin{itemize}
\item
First, we will use the Gauss-Codazzi formula to deduce a control on how the mean curvature differs from its average.
\item
We will then apply the same procedure to obtain an analogous control of the conformal Gauss map. 
\item
Using suitable conformal transformations, we will show that the average of the mean curvature can be set to zero.
\item
Finally, we will show that under a small energy condition, the average of the mean curvature can be cancelled, which ultimately yields the desired estimates.
\end{itemize}

\subsection{Controlling \texorpdfstring{$H$}{TEXT}}
\begin{prop}
\label{110320201432}
Let $\bp \in \mathcal{E}\left(\D \right)$  be a conformal weak Willmore immersion. 
Assume 
\begin{equation}
\label{ebound}
\int_{\D} \left| \nabla \n \right|^2 \le \frac{4\pi}{3},
\end{equation}
and
\begin{equation}
\label{lbound}
\left\| \nabla \lambda \right\|_{L^{2,\infty} \left( \D \right) } \le C_0,
\end{equation}
for some constant $C_0>0$. \\
There exists $C>0$ depending only on $C_0$ such that for any $U \in W^{1,2}_0\cap L^\infty \left( \D_{\frac{1}{2}} , \R^2 \right)$, we have
\begin{equation} \label{Hest}
\int_{\D_{\frac{1}{2}}} \langle \nabla H ,U\rangle e^{2\lambda} dxdy \;\leq\; C \left\|\Ar e^{-\lambda} \right\|_{L^2 \left( \D_{\frac{1}{2}} \right)}  \left\| \nabla U e^{\lambda}\right\|_{L^2 \left(\D_{\frac{1}{2}} \right)}\:.
\end{equation}
\end{prop}
\begin{proof}
Using Lemma II.1 in \cite{BernardRiviereSubcritical}, we know that
\begin{equation}\label{BR}
\Vert\nabla\lambda\Vert_{L^2(\D_{\frac{1}{2}})}+\left|\lambda-\dfrac{1}{|\D_{\frac{1}{2}}|}\int_{\D_{\frac{1}{2}}}\lambda(x,y)\,dxdy \right|\;\leq\;C(\Vert\nabla\lambda\Vert_{L^{2,\infty}(\D)},\Vert\nabla\vec{n}\Vert_{L^2(\D)})\:.
\end{equation}
In particular, there exists another constant $ C >0$ depending only on $C_0$, such that on $\D_{\frac{1}{2}}$ one has
\begin{equation}
\label{elbound}
\frac{e^{\bar{\lambda}}}{C} \le e^{\lambda} \le C e^{\bar{\lambda}} \hbox{ on } \D_{\frac{1}{2}}, 
\end{equation}
where $\displaystyle \bar{\lambda}= \frac{1}{|\D_{\frac{1}{2}}|}\int_{\D_{\frac{1}{2}}}\lambda\,dxdy$.\\
Next, given $U=(U_1,U_2) \in\R^2\otimes W_0^{1,2}( \D_{\frac{1}{2}})$, we have
\begin{equation}
\label{050320201036}
\int_{\D_{\frac{1}{2}}} \langle \nabla H, U\rangle e^{2\lambda} dxdy = \int_{\D_{\frac{1}{2}}}  \left( H_x U_1 + H_y U_2 \right)e^{2\lambda}dx dy. \\
\end{equation}
Let us now recall the Gauss-Codazzi formula (see \eqref{Gausscodazziformereelle}) below
\begin{equation}
\label{05032020}
\left\{
\begin{aligned}
  e^{2\lambda}H_x &= \left(\frac{l-n}{2}\right)_x+m_y \\
  e^{2\lambda}H_y &= -\left(\frac{l-n}{2}\right)_y+m_x
\end{aligned}
\right.
\end{equation}
where $A =  \begin{pmatrix} l & m \\ m & n \end{pmatrix}$.

Injecting \eqref{05032020} into \eqref {050320201036}, one finds
\begin{equation}
\label{050320201047}
\begin{aligned}
\int_{\D_{\frac{1}{2}}} \langle \nabla H, U\rangle e^{2\lambda}  dxdy &= \int_{\D_{\frac{1}{2}}} \left( \frac{l-n}{2}\right)_x U_1 +m_y U_1 - \left( \frac{l-n}{2}\right)_y U_2 + m_x U_2 \\
&= \int_{\D_{\frac{1}{2}}}  \nabla \left( \frac{l-n}{2}\right)\cdot \begin{pmatrix} U_1 \\ -U_2 \end{pmatrix} + \nabla m \cdot  \begin{pmatrix} U_2 \\ U_1 \end{pmatrix} \\
&= - \int_{\D_{\frac{1}{2}}} \left( \frac{l-n}{2}\right) \mathrm{div} \begin{pmatrix} U_1 \\ -U_2 \end{pmatrix} + m \mathrm{div}  \begin{pmatrix} U_2 \\ U_1 \end{pmatrix},
\end{aligned}
\end{equation}
since $U$ vanishes on  $\partial \D_{\frac{1}{2}}$.
We can then deduce: 
\begin{equation}
\label{050320201109}
\begin{aligned}
\left| \int_{\D_{\frac{1}{2}}} \langle \nabla H, U\rangle e^{2\lambda} dxdy \right| &\le C \left( \left\| \frac{l-n}{2} \right\|_{L^2 \left( \D_{\frac{1}{2}} \right)} +  \left\| m \right\|_{L^2 \left(\D_{\frac{1}{2}}\right)} \right) \left\| \nabla U\right\|_{L^2 \left( \D_{\frac{1}{2}}\right)} \\
&\le C \left\|\Ar  \right\|_{L^2 \left( \D_{\frac{1}{2}} \right)}  \left\| \nabla U \right\|_{L^2 \left(\D_{\frac{1}{2}} \right)}\:. \\
\end{aligned}
\end{equation}
The conclusion follows from (\ref{elbound}).
\end{proof}

We will use in a decisive way the following result from J. Bourgain and H. Brezis (Theorem 3' in \cite{BB}):
\begin{theo*}
Let $\Omega\subset\mathbb{R}^2$ be a bounded set with Lipschitz boundary. If $f\in L^2(\Omega)$ has null average on $\Omega$, then there exists $V\in \R^2\otimes(W_0^{1,2}\cap L^\infty)(\Omega)$ such that
$$
\mathrm{div}\,V=\;f\:.
$$
Moreover
\begin{equation}\label{BB0}
\Vert V\Vert_{L^\infty(\Omega)}+\Vert\nabla V\Vert_{L^2(\Omega)}\;\leq\;C(\Omega)\Vert f\Vert_{L^{2}(\Omega)}\:.
\end{equation}
\end{theo*}
Let 
$$
\overline{H}\;:=\;\frac{1}{\left|\D_{\frac{1}{2}}\right|}\int_{\D_{\frac{1}{2}}}H\,dxdy\: .
$$
Applying the above theorem and (\ref{elbound}), we know that there exists $V\in \R^2\otimes(W_0^{1,2}\cap L^\infty)(\D_{\frac{1}{2}})$ such that
$$
\mathrm{div}\,V\;=\;H-\overline{H} \qquad\text{on}\:\:\D_{\frac{1}{2}}
$$
and
\begin{equation}\label{BB}
\Vert V\Vert_{L^\infty(\D_{\frac{1}{2}})}+\Vert\nabla V\Vert_{L^2(\D_{\frac{1}{2}})}\;\leq\;Ce^{-\bar{\lambda}}\Vert (H-\overline{H}) e^{\lambda} \Vert_{L^{2}(\D_{\frac{1}{2}})}\:.
\end{equation}

Now replacing $H$ by $H-\overline{H}$ in \eqref{050320201109}, setting $U=e^{-2\lambda}V$, and integrating by parts yields
\begin{eqnarray}\label{crux}
 \int_{\D_{\frac{1}{2}}} |H-\overline{H}|^2 dxdy&=&\left| \int_{\D_{\frac{1}{2}}} \langle H-\overline{H}, \mathrm{div}(V)\rangle dxdy \right|\;\;=\;\;\left| \int_{\D_{\frac{1}{2}}} \langle \nabla H, U\rangle e^{2\lambda}\,dxdy \right|\nonumber\\
&\leq& C \left\|\Ar e^{-\lambda} \right\|_{L^2 \left( \D_{\frac{1}{2}} \right)}  \left\| e^{\lambda}\nabla U \right\|_{L^2 \left(\D_{\frac{1}{2}} \right)}\:.
\end{eqnarray}
But by (\ref{BR}), (\ref{elbound}) and (\ref{BB}), we have
\begin{eqnarray*}
\left\| e^{\lambda}\nabla U\right\|_{L^2 \left(\D_{\frac{1}{2}}\right)} & \leq&\; e^{-\bar{\lambda}}\left(\Vert \nabla V\Vert_{L^2\left(\D_{\frac{1}{2}}\right)}+\Vert\nabla\lambda\Vert_{L^2\left(\D_{\frac{1}{2}}\right)}\Vert V\Vert_{L^\infty\left(\D_{\frac{1}{2}}\right)}\right)\\
&\leq&\;C e^{-2\bar{\lambda}} \Vert (H-\overline{H}) e^{\lambda} \Vert_{L^{2}(\D_{\frac{1}{2}})},
\end{eqnarray*}
Introducing this into (\ref{crux}) and applying once more \eqref{elbound} yields

\begin{cor}
    \label{corf}
    Let $\bp \in \mathcal{E}\left(\D \right)$  be a conformal weak Willmore immersion. 
    Assume 
    \begin{equation}
    \label{ebound}
    \int_{\D} \left| \nabla \n \right|^2 \le \frac{4\pi}{3},
    \end{equation}
    and
    \begin{equation}
    \label{lbound}
    \left\| \nabla \lambda \right\|_{L^{2,\infty} \left( \D \right) } \le C_0,
    \end{equation}
for some constant $C_0>0$. \\
    There exists $C>0$ depending only on $C_0$ such that 
    \begin{equation} \label{Hest}
    \Vert (H -\overline{H}) e^{\lambda}\Vert_{L^2 \left( \D_{\frac{1}{2}} \right)} \;\leq\; C \left\|\Ar e^{-\lambda} \right\|_{L^2 \left( \D_{\frac{1}{2}} \right)}\:.
    \end{equation}
\end{cor}
    
\subsection{Controlling \texorpdfstring{$Y$}{TEXT}}
In this subsection we will abundantly use the conformal Gauss map $Y$, and  assume \eqref{lbound}  and \eqref{elbound} hold on the whole of $\D$, for notational convenience. In addition, up to a translation, we may assume that $\bp(0)={0}$ and up to a dilation that $\bar{\lambda} =0$. Hence, 
\begin{equation}
\label{elbound2}
\exists\:\: C(C_0) \text{ s.t. } \frac{1}{C} \le e^\lambda \le C,
\end{equation}
and
\begin{equation}
\label{Phiinf}
\left\| \bp \right\|_{W^{1,\infty} \left( \D \right)} \le C.
\end{equation}
\begin{prop}
\label{120320200910}
Let $\bp \in \mathcal{E}\left(\D \right)$  be a conformal weak Willmore immersion.
Assume \eqref{elbound2} and \eqref{Phiinf} hold. \newline
Then, there exists $C >0$ depending only on $C_0$ such that:
\begin{equation} \label{120320200914} 
\left\| (Y- \overline{Y})e^{\lambda} \right\|_{L^2 \left( \D \right) } \le C \left\| \Ar e^{-\lambda}\right\|_{L^2 \left( \D \right)} .\end{equation}
\end{prop}
\begin{proof}
Let $U \in W^{1,2}_0\cap L^\infty \left( \D , \R^5 \otimes \R^2 \right)$. Then, owing to \eqref{230520200910}, we have
$$
\begin{aligned}
\left| \int_{\D}  \langle \nabla Y,U\rangle e^{2\lambda}\,dxdy \right| &\le \left| \int_{\D}  \left\langle \nabla H ,\left\langle \begin{pmatrix} \Phi \\ \frac{ | \Phi|^2 -1}{2} \\ \frac{ |\Phi|^2 + 1}{2} \end{pmatrix}, U \right\rangle\right\rangle e^{2\lambda}\,dxdy \right| +C \left\| \Ar e^{-\lambda} \right\|_{L^2 \left(\D \right) } \left\| U \right\|_{L^2\left( \D \right)} \\
&\le C \left\| \Ar e^{-\lambda} \right\|_{L^2\left( \D \right)} \left( \left\| e^{\lambda} \nabla \left( \left\langle \begin{pmatrix} \Phi \\ \frac{ | \Phi|^2 -1}{2} \\ \frac{ |\Phi|^2 + 1}{2} \end{pmatrix}, U \right\rangle \right) \right\|_{L^2 \left( \D \right)} +  \left\| U \right\|_{L^2\left( \D \right)} \right) \\
& \le C \left\| \Ar e^{-\lambda} \right\|_{L^2\left( \D \right)} \left\| U \right\|_{W^{1,2} \left(\D \right) },
\end{aligned}$$ 
where we have used \eqref{elbound2} and estimate \eqref{Phiinf}. The proof then proceeds as in corollary \ref{corf}, by setting $ \mathrm{div} \,V  =  Y - \overline{Y}$ and $U = e^{-2 \lambda} V$, so as to obtain the announced estimate:
$$
\int_{\D} \left|Y- \overline{Y} \right|^2 e^{2\lambda}\; dxdy \le C \left\| \Ar e^{-\lambda} \right\|^2_{L^2 \left( \D \right)} \:.
$$
\end{proof}

Note that
$$
\begin{aligned}
\int_{\D} \left\langle Y-\overline{Y}, Y-\overline{Y} \right\rangle_{4,1} \, dxdy &= \int_{\D} \left( \left\langle Y,Y \right\rangle_{4,1}  -2 \left\langle Y,\overline{Y} \right\rangle_{4,1} + \left\langle \overline{Y}, \overline{Y}\right\rangle_{4,1} \right)  \, dxdy \\
&= \left| \D \right| -2 \left| \D \right| \left\langle \overline{Y}, \overline{Y} \right\rangle_{4,1} +   \left| \D \right| \left\langle \overline{Y}, \overline{Y} \right\rangle_{4,1} \\
&= \left|\D \right| \left( 1-  \left\langle \overline{Y}, \overline{Y} \right\rangle_{4,1} \right),
\end{aligned}
$$
where we have used that $\left\langle Y,Y \right\rangle_{4,1}=1$. It then follows that
\begin{equation}
\label{120320201054}
\left\langle \overline{Y}, \overline{Y} \right\rangle_{4,1}  = 1 - \frac{1}{ \left|\D \right| }\int_{\D} \left\langle Y-\overline{Y}, Y-\overline{Y} \right\rangle_{4,1} \, dxdy.
\end{equation}

\begin{cor}
\label{corY}
Under the hypotheses of proposition \ref{120320200910}, there exists $C>0$ depending only on $C_0$ such that: 
$$\left\langle \overline{Y}, \overline{Y} \right\rangle_{4,1} \ge 1 - C \left\| \Ar e^{-\lambda} \right\|^2_{L^2 \left( \D \right)}.$$
Hence, there exists $\varepsilon_0>0$ depending only on $C_0$ such that, if
$$ \left\| \Ar e^{-\lambda} \right\|_{L^2 \left( \D \right)} \le \varepsilon_0,$$ then
$$\left\langle \overline{Y}, \overline{Y} \right\rangle_{4,1} \ge \frac{1}{2}.$$

\end{cor}
\begin{proof}
The result is readily obtained from combining proposition \ref{120320200910} with \eqref{120320201054}. In order to reintroduce the conformal factor, one calls upon \eqref{elbound2}. Moreover, one uses that for any $v \in \R^{4,1}$, it holds
$$\left\langle v,v\right\rangle_{4,1} \le \left\langle v,v\right\rangle.$$
\end{proof}

\subsection{Cancelling \texorpdfstring{$\overline{H}$}{TEXT}}
\begin{prop}
\label{200320201141}
Let $\bp \in \mathcal{E}\left(\D \right)$  be a conformal weak Willmore immersion satisfying \eqref{ebound} and \eqref{lbound}, while \eqref{elbound} holds on the whole disk $\D$. Then, there exists $\varepsilon_0>0$ depending only on $C_0$ such that if $$ \left\| \Ar e^{-\lambda} \right\|_{L^2 \left( \D \right)} \le \varepsilon_0,$$ then  there exists $\Theta \in \mathrm{Conf}\left( \R^3 \right)$ such that $\bpsi := \Theta \circ \bp$ still satisfies \eqref{lbound}  and \eqref{elbound}, as well as 
$$\overline{H}_\bpsi= {0}.$$
\end{prop}
\begin{proof}
Without loss of generality, as previously seen, we can arrange for \eqref{elbound2} and \eqref{Phiinf} to hold (instead of merely  \eqref{ebound}). To do so, it suffices to apply suitable translation and dilation. For notational simplicity, the resulting immersion will continue to be denoted by $\bp$. \\

Let $$Y = H \begin{pmatrix} \bp \\ \frac{|\bp|^2-1}{2} \\\frac{|\bp|^2+1}{2}\end{pmatrix} + \begin{pmatrix} \n \\ \langle \n, \bp \rangle \\ \langle \n, \bp \rangle\end{pmatrix} =: \begin{pmatrix} \bY_{123}, Y_4,Y_5 \end{pmatrix}$$
be the conformal Gauss map of $\bp$. We adopt the same notation as in \cite{bibnmconfgaussmap}, with $\bY_{123}(x,y) \in \R^3$, and $Y_4(x,y), Y_5(x,y) \in \R$.\\
According to (20) of \cite{bibnmconfgaussmap}, one recovers $H$ from $Y$ via \begin{equation} \label{120320201245} H = Y_5-Y_4. \end{equation} 
Theorem 2.4 of \cite{bibnmconfgaussmap} ensures that a conformal transformation acting on $\bp$ induces on $Y$ a change by a matrix $M \in SO(4,1)$. More precisely, for an inversion $\bpsi = \frac{ \bp-\vec{a}}{\left|\bp- \vec{a} \right|^2}$ about a point $\vec{a}\in\R^3$, we find $Y_{\bpsi} = M Y$ with 
\begin{equation}
\label{230520200815} M = \begin{pmatrix} -Id & 0 &0 \\ 0&1&0 \\0&0&-1 \end{pmatrix}\begin{pmatrix} Id & \vec{a} &-\vec{a} \\- \vec{a}^T&1- \frac{|a|^2}{2}&\frac{|a|^2}{2} \\-\vec{a}^T&-\frac{|\vec{a}|^2}{2}&1+\frac{|\vec{a}|^2}{2} \end{pmatrix} = \begin{pmatrix} -Id & -\vec{a} &\vec{a} \\ -\vec{a}^T&1- \frac{|\vec{a}|^2}{2}&\frac{|\vec{a}|^2}{2} \\\vec{a}^T&\frac{|\vec{a}|^2}{2}&-1-\frac{|\vec{a}|^2}{2} \end{pmatrix}. 
\end{equation}
Hence
$$Y_\bpsi = \begin{pmatrix} -\bY_{123} +\vec{a}(Y_5-Y_4) \\ -\langle \vec{a}, \bY_{123} \rangle  + Y_4 + \frac{|\vec{a}|^2}{2} \left( Y_5-Y_4 \right) \\ \langle \vec{a}, \bY_{123} \rangle - Y_5 -\frac{|\vec{a}|^2}{2}\left( Y_5-Y_4 \right) \end{pmatrix}.$$
This yields that
$H_\bpsi = 2\langle \vec{a}, \bY_{123} \rangle - Y_5 -Y_4 - |\vec{a}|^2 \left( Y_5 - Y_4 \right)$, or in other symbols,
\begin{equation}
\label{130320201349}
\overline{H}_{\bpsi} = \langle \vec{a}, \overline{\bY}_{123} \rangle - \overline{Y}_5 -\overline{Y}_4 - |\vec{a}|^2 \left( \overline{Y}_5 - \overline{Y}_4 \right).
\end{equation}

Note that if $\overline{H}=\overline{Y}_5-\overline{Y}_4= 0$, the result is immediate by choosing $\Theta=Id$. We will thus assume without loss of generality that $\overline{H}\ne0$. 

\vspace{0.5cm}

We decompose $\vec{a} = x \frac{\overline{\bY}_{123}}{\left|\overline{Y}_{123}\right|} +y\vec{v}_1 +z\vec{v}_2$ where $\left\{ \frac{\overline{\bY}_{123}}{\left|\overline{\bY}_{123}\right|}, \vec{v}_1, \vec{v}_2 \right\}$ is an orthonormal basis of $\R^3$.  From \eqref{130320201349} we then deduce
$$\begin{aligned}
\overline{H}_\bpsi &=2 x \left|\overline{\bY}_{123}\right| - \overline{Y}_5 -\overline{Y}_4  - (x^2 + y^2 +z^2) \left( \overline{Y}_5 - \overline{Y}_4 \right) \\
&= - \left( \overline{Y}_5 - \overline{Y}_4 \right) \left( \left[ x- \frac{ \left|\overline{\bY}_{123}\right|}{ \overline{Y}_5 - \overline{Y}_4} \right]^2 + y^2 + z^2 - \frac{ \left|\overline{\bY}_{123}\right|^2}{ \left(\overline{Y}_5 - \overline{Y}_4\right)^2 }+ \frac{\overline{Y}_4 + \overline{Y_5}}{\overline{Y}_5 - \overline{Y}_4 } \right) \\
&=  - \left( \overline{Y}_5 - \overline{Y}_4 \right) \left( \left[ x- \frac{ \left|\overline{\bY}_{123}\right|}{ \overline{Y}_5 - \overline{Y}_4} \right]^2 + y^2 + z^2 - \frac{ \left|\overline{\bY}_{123}\right|^2+\overline{Y}_4^2 - \overline{Y_5}^2 }{\left( \overline{Y}_5 - \overline{Y}_4 \right)^2} \right) \\
&=- \left( \overline{Y}_5 - \overline{Y}_4 \right) \left( \left[ x- \frac{ \left|\overline{\bY}_{123}\right|}{ \overline{Y}_5 - \overline{Y}_4} \right]^2 + y^2 + z^2 - \frac{ \left\langle \overline{Y} , \overline{Y} \right\rangle_{4,1} }{\left( \overline{Y}_5 - \overline{Y}_4 \right)^2} \right)
\end{aligned}$$
Thus there exists $\vec{a}$ such that  $\overline{H}_\bpsi = 0$ if and only if $\left\langle \overline{Y} , \overline{Y} \right\rangle_{4,1} \ge 0$, and in that case any $\vec{a}$ belonging to the sphere $\overline{\s}_Y$ of center $ \frac{ \overline{\bY}_{123}}{ \overline{Y}_5 - \overline{Y}_4} \in \R^3$ and radius $\sqrt{\frac{ \left\langle \overline{Y} , \overline{Y} \right\rangle_{4,1} }{\left( \overline{Y}_5 - \overline{Y}_4 \right)^2}}$ satisfies $\overline{H}_\bpsi = 0$.
Owing to corollary \ref{corY}, we know there exists $\varepsilon_0>0$ depending only on $C_0$ such that if $\left\| \Ar e^{-\lambda}\right\|_{L^2\left(\D \right)} \le \varepsilon_0$ then $ \left\langle \overline{Y}, \overline{Y} \right\rangle_{4,1} \ge \frac{1}{2}$. Thus indeed there exists $\vec{a} \in \R^3$ cancelling $\overline{H}_\bpsi$.
\vspace{0.5cm}
A straightforward computation shows that, if we denote by $\lambda_\bpsi$ the conformal factor of $\bpsi$, then we have
\begin{equation}
  \label{logest}
  \lambda_\bpsi = \lambda + \log  |\bp- \vec{a} |.
\end{equation}
Since $\log$ has its gradient in the weak Marcinkiewicz space $L^{2,\infty}$, from (\ref{Phiinf}), it follows that $\lambda_\bpsi$ inherits the bound \eqref{lbound} on $\lambda$ so soon as there exists a constant $C>0$ depending only on $C_0$ such that $\frac{1}{C} \le |\bp- \vec{a} | \le C$. This  ensures that (\ref{lbound}) holds for $\lambda_\bpsi$ and enables us to conclude the proof. We now establish the existence of such a constant $C$. \\

\noindent
\textbf{Claim.}
We can choose $\vec{a}\in \overline{\s}_Y$ such that 
\begin{equation}\label{bnp}
\frac{1}{C} \leq |\bp- \vec{a} | \leq C,
\end{equation}
 with $C$ depending only on $C_0$.\\
{\it Proof.} By hypothesis, $\bp(\D) \subset B(0,R_0)$ for some $R_0$ depending only on $C_0$. We prove there exists $R_0'\geq 2R_0$ depending only on $C_0$ such that  
$$ B(0,R_0') \cap \overline{\s}_Y \neq \emptyset.$$
In doing so, we first establish that there exists $C>0$ depending only on $C_0$, such that
\begin{equation}
  \label{Yuni}
  \left\Vert \overline{Y} \right\Vert \le C,
\end{equation}
where $\Vert\cdot\Vert$ is the Euclidean metric.\\
Combining \eqref{ebound} and \eqref{Phiinf}, one finds: 
\begin{equation}
\label{190320201059}
\left\| H\right\|_{L^2 \left(\D \right)}+ \left\| \begin{pmatrix} \bp \\ \frac{ |\bp|^2-1}{2} \\ \frac{ |\bp|^2+1}{2} \end{pmatrix} \right\|_{L^2 \left(\D \right)} +\left\| \begin{pmatrix} \n  \\ \langle \n, \bp \rangle \\ \langle \n , \bp \rangle \end{pmatrix} \right\|_{L^2 \left(\D \right)} \le C.
\end{equation}
Given definition \eqref{230520200727} for the conformal Gauss map, one has:
\begin{equation}
\label{190320201109}
\left\| Y \right\|_{L^2 \left(\D \right)} \le C,
\end{equation}
and consequently we deduce (\ref{Yuni}).\\

Next, let $S(c,r):=\overline{\s}_Y$. We have 
$$\begin{aligned}
\Vert c\Vert-r&= \frac{\vert \overline{\bY}_{123}\vert}{\vert \overline{Y}_5 - \overline{Y}_4\vert} - \sqrt{\frac{ \left\langle \overline{Y} , \overline{Y} \right\rangle_{4,1} }{\left( \overline{Y}_5 - \overline{Y}_4 \right)^2}}=\frac{\vert \overline{\bY}_{123}\vert-\sqrt{\left\langle \overline{Y} , \overline{Y} \right\rangle_{4,1} }}{\vert \overline{Y}_5 - \overline{Y}_4\vert}\\
&=\frac{\vert \overline{\bY}_{123}\vert^2-\left\langle \overline{Y} , \overline{Y} \right\rangle_{4,1} }{\vert \overline{Y}_5 - \overline{Y}_4\vert \left(\vert \overline{\bY}_{123}\vert+\sqrt{\left\langle \overline{Y} , \overline{Y} \right\rangle_{4,1} }\right)}\\
&=\pm \frac{\vert \overline{Y}_{4} + \overline{Y}_{5}\vert}{\vert \overline{\bY}_{123}\vert+\sqrt{\left\langle \overline{Y} , \overline{Y} \right\rangle_{4,1} }}
\end{aligned}
$$ 
Calling upon \ref{corY} and (\ref{Yuni}), we obtain the desired result with 

$$R_0'=\max (2R_0,2\vert \overline{Y}_4+\overline{Y}_5\vert^\frac{1}{2}).$$

\noindent
It is important to note that the radius of $\overline{\s}_Y$ is bounded from below thanks to  (\ref{Yuni}) and corollary \ref{corY}.\\

Clearly, if $\overline{\s}_Y \not\subset B(0,R_0')$, then any $\vec{a}\in S(0,R_0')\cap \overline{\s}_Y$ satisfies the desired property \eqref{bnp}. On the other hand, suppose that $\overline{\s}_Y \subset B(0,R_0')$. The smallness hypothesis on $\left\| \nabla \n \right\|^2_{L^2}$ warrants the impossibility to construct a sequence $\{\bp_k\}_k$ satisfying \ref{ebound} and for which $\bp_k(\D)$ gets arbitrarily close to $\overline{\s}_Y$ (whose radius is bounded from below, as remarked above). Accordingly, there exists $\vec{a}$ whose distance from $\bp(\D)$ is bounded from below (and, of course, also bounded from above). This concludes the proof of the claim and thus of the proposition.
\end{proof}

\begin{remark}
It might seem strange to favor working with the Euclidean average, rather than with the average computed against the metric. 
Equality \eqref{130320201349} reveals why: the exchange law enables recovering the mean curvature after the conformal transformation.  The metric change induced by the transformation modifies the exchange law into a less convenient form.
\end{remark}
\begin{remark}
The fact that $\varepsilon_0$ depends only on $C_0$ stems from the dilation used to neutralize the constant $\bar{\lambda}$ in \eqref{elbound}. This dilation is not merely used to simplify the problem, but it is also part of the resulting conformal transformation.
\end{remark}

\subsection{\texorpdfstring{$\varepsilon$}{TEXT}-regularity for \texorpdfstring{$\Ar$}{TEXT}}

\begin{theo}
\label{200320201414}
Let $\bp \in \mathcal{E}\left(\D \right)$  be a conformal weak Willmore immersion. Assume 
$$
\int_{\D} \left| \nabla \n \right|^2 \le \frac{4\pi}{3},
$$
and
$$
\left\| \nabla \lambda \right\|_{L^{2,\infty} \left( \D \right) } \le C_0,
$$
for some constant $C_0>0$. \\
Then there exists $\varepsilon_0(C_0)$  such that, if $$\left\| \Ar e^{-\lambda} \right\|_{L^2\left( \D \right)} \le \varepsilon_0,$$ there exists $C(r, C_0,\rho) >0$ with
$$ \left\| \Ar e^{-\lambda} \right\|_{L^\infty \left( \D_{\frac{1}{2}} \right) } \le C \left\| \Ar e^{-\lambda} \right\|_{L^2 \left( \D \right)}.$$
\end{theo}
\begin{proof}

Owing to corollary 2.2 of \cite{bibnmheps}, we know that $\bp$ satisfies \eqref{elbound} on  $\D_{\frac{3}{4} }$. By hypothesis, $\bp$ satisfies \eqref{ebound} and \eqref{lbound} on $\D$ and \emph{a fortiori} on $\D_{\frac{3}{4} }$. We can then apply proposition \ref{200320201141} on $\D_{\frac{3}{4} }$, assuming $\varepsilon_0 $ 
small enough
, so as to find $\Theta \in \mathrm{Conf}(\R^3)$ such that $\bpsi = \Theta \circ \bp$ still satisfies \eqref{elbound}, and such that $\overline{H}_\bpsi =0$ (with the average taken over $\D_{\frac{3}{4} }$). \\
Since $\bpsi$ satisfies  \eqref{elbound}, we can apply proposition \ref{110320201432} and its corollaries (specifically \ref{corf}) on $\D_{\frac{3}{4} }$  to obtain, from 
\eqref{Hest} and $\overline{H}_\bpsi =0$ that
\begin{equation} \label{200320201333} \left\| \left( H_\bpsi-\overline{H}_\bpsi \right) e^{\lambda_\bpsi} \right\|_{L^2 \left(\D_{\frac{3}{4}} \right)} =\left\|  H_\bpsi e^{\lambda_\bpsi} \right\|_{L^2 \left(\D_{\frac{3}{4}} \right)}   \le C \| \Ar_\bpsi e^{-\lambda_\bpsi} \|_{L^2 \left( \D_{\frac{3}{4}} \right) }.\end{equation}
Thus, \begin{equation} \label{200320201338} \left\| \nabla \n_\bpsi \right\|_{L^2 \left(\D_{\frac{3}{4} } \right)} \le  \left\| \left( H_\bpsi-\overline{H}_\bpsi \right) e^{\lambda_\bpsi} \right\|_{L^2 \left(\D_{\frac{3}{4}} \right)}  + \| \Ar_\bpsi e^{-\lambda_\bpsi} \|_{L^2 \left( \D_{\frac{3}{4}} \right) } \le C  \| \Ar_\bpsi e^{-\lambda_\bpsi} \|_{L^2 \left( \D_{\frac{3}{4}} \right) }.\end{equation}
Choosing $\varepsilon_0  = \mathrm{min} \left( \tilde \varepsilon_0, \frac{4 \pi}{3C(C_0 )} \right)$, with $C(C_0 )$ the final constant in \eqref{200320201338}, we find that $\bpsi$ satisfies 
\begin{equation} \label{200320201346} \left\| \nabla \n_\bpsi \right\|_{L^2 \left(\D_{\frac{3}{4}} \right)} \le \frac{4\pi}{3}.\end{equation}
The classical $\varepsilon$-regularity for Willmore immersions \cite{bibanalysisaspects} states that 
\begin{equation}
\label{200320201350}
\left\| \nabla \n_\bpsi \right\|_{L^\infty \left( \D_{\frac{1}{2}} \right) } \le C \left\| \nabla \n_\bpsi \right\|_{L^2 \left(\D_{\frac{3}{4} } \right)}.
\end{equation}
Combining \eqref{200320201338} and \eqref{200320201350}, we deduce 
$$
\left\| \Ar_\bpsi e^{-\lambda_\bpsi} \right\|_{L^\infty \left( \D_{\frac{1}{2}} \right) } \le \left\| \nabla \n_\bpsi \right\|_{L^\infty \left( \D_{\frac{1}{2}} \right) } \le C \left\| \nabla \n_\bpsi \right\|_{L^2 \left(\D_{\frac{3}{4}} \right)} \le C  \| \Ar_\bpsi e^{-\lambda_\bpsi} \|_{L^2 \left( \D_{\frac{3}{4}} \right)},$$
which yields
\begin{equation}
\label{200320201354}
\left\| \Ar_\bpsi e^{-\lambda_\bpsi} \right\|_{L^\infty \left( \D_{\frac{1}{2}} \right) } \le C \| \Ar_\bpsi e^{-\lambda_\bpsi} \|_{L^2 \left( \D \right)}.
\end{equation}
Since \eqref{200320201354} is conformally invariant, it holds with $\bp = \Theta^{-1} \circ \bpsi$ in place of $\bpsi$, which concludes the proof.
\end{proof}

\subsection{Proof of theorem \ref{mainth2}}
We end this section by removing the small bound on the total curvature. As we have seen, this bound is decisive in controling the conformal factor. When the conformal factor can no longer be controlled, there must be concentration of energy and a bubbling phenomenon ensues. However, owing to our hypotheses, we will see that all bubbles must be round spheres. In that case, the conformal factor still satisfies some Harnack estimate, which, as we will see, is sufficient to conclude.\\

The proof of Theorem \ref{mainth2} goes in $4$ steps:
\begin{itemize}
\item We first show that the statement fails only when bubbling develops.
\item We prove all bubbles must be Euclidean spheres.
\item We eliminate those bubbles with the help of an inversion.
\item Finally, conformal invariance leads to a contradiction.
\end{itemize}

\begin{proof}
For the sake of contradiction, consider a sequence $\bp_k  \, : \, \D \rightarrow \R^3$ such that $\bp_k \in \mathcal{E}(\D)$ is a conformal Willmore immersion satisfying \eqref{Mest1} and \eqref{Mest2}. We further assume that the induced conformal classes lie in a compact subset of Moduli space, and that there exists $C(k)\rightarrow \infty$ such that:
\begin{equation}
\label{HN1}
\left\| \Ar_k e^{-\lambda_k} \right\|_{L^\infty \left( \D_{\frac{1}{2} } \right) } \ge C(k) \left\| \Ar_k e^{-\lambda_k} \right\|_{L^2 \left( \D \right) }.
\end{equation}
Up to a dilation, we can also assume that there exists $A>0$ such that $\bp_k (\D) \subset B_A(0)$ and $\bp_k \left(\D_{\frac{1}{2}} \right) \cap B_{\frac{1}{A} } (0)^c \neq \emptyset$.
We are then precisely in the situation of theorem I.3 of \cite{bibenergyquant}, which states that there exist a finite number $N$ of radii $\rho^i_k \rightarrow 0$ and points $a^i_k \rightarrow a^i \in \D$, a  Willmore immersion $\bp_\infty : \D \rightarrow \R^3$, and some possibly branched Willmore immersions $\omega^i :  \s^2 \rightarrow \R^3$, as well as  conformal transformations $\theta_k$, $\xi^i_k$, such that:
\begin{equation}
\label{ri1}
\begin{aligned}
&\theta_k \circ \bp_k \rightarrow \bp_\infty \quad C^\infty_{\mathrm{loc} } (\D \backslash \{ a^1, \dots, a^{t} \} )\\
&\xi^i_k \circ \bp_k ( \rho^i_k x + a^i_k ) \rightarrow \omega^i \quad C^\infty_{\mathrm{loc}} \left( \R^2 \backslash \{ \text{ finite set } \} \right) \\
& \left\| \nabla \n_k \right\|^2_{L^2 \left( \D \right) } \rightarrow \left\| \nabla \n_\infty \right\|^2_{L^2 \left( \D \right) } + \sum_{i=1}^N \left\| \nabla \n_{\omega^i} \right\|^2_{L^2 \left( \R^2 \right) } \\
& \left\| \Ar_k e^{-\lambda_k} \right\|^2_{L^2 \left( \D \right) } \rightarrow \left\| \Ar_\infty e^{-\lambda_\infty} \right\|^2_{L^2 \left( \D \right) } + \sum_{i=1}^N \left\| \Ar_{\omega^i} e^{-\lambda_{\omega^i}} \right\|^2_{L^2 \left( \R^2 \right) }.
\end{aligned}
\end{equation}
In \cite{bibenergyquant}, it is in fact $ \Vert H\Vert_{L^2}^2$ which is quantized, but as remarked in lemma 3.1 of \cite{bibnm3}, one has also quantization for the full second fundamental form and in particular for its traceless part as well.\\

In our case, there is at least one concentration point inside $\D_\frac{3}{4}$, else we would simply conclude by covering $\D_{\frac{1}{2}}$ with a finite number of disks of radius bounded from below and satisfying the hypothesis of theorem \ref{mainth1}. We then see that the energies of the bubble and of the limit are controlled, namely
\begin{equation} \label{ri2} \begin{aligned}
\left\| \Ar_\infty e^{-\lambda_\infty} \right\|_{L^2 (\D)} &\le \varepsilon_0 \\
\forall\: i \quad  \left\| \Ar_{\omega^i} e^{-\lambda_{\omega^i}} \right\|_{L^2 \left( \s^2 \right) } &\le \varepsilon_0.
\end{aligned} \end{equation}
Owing to \eqref{ri2}, we see that each bubble is a round sphere. Indeed Theorem H in \cite{michelatclassifi} (see also lemma \ref{mini} in the Appendix) guarantees that the bubbles, even if branched, are conformal inversions of minimal surfaces. But a classical result states that the total curvature of a minimal surface is a multiple of $4\pi$. Hence, assuming that $\varepsilon_0$ is small enough ensures that the bubbles are round spheres.

From the proof of theorem 0.2 of \cite{MR3843372}, it is known that a round sphere cannot be glued onto a compact surface without a third surface appearing in between, and this surface is necessarily non-umbilic\footnote{The argument is as follows: between the round sphere and the compact piece, there is a small geodesic circle. Blowing up the surface around this geodesic gives rise to a non compact Willmore surface with at least two ends which cannot be umbilic. Hence all the involved concentration points develop only one simple bubble which is a round sphere and $\bp_\infty$ must be constant.}. Hence all the round bubbles are simple,  and the concentration points must be isolated \footnote{In fact there is only one concentration point since the argument of theorem 0.2 of \cite{MR3843372} applies between two bubbles.}.\\

In Willmore bubbling, singular points (branched or non-compact) can only appear at concentration points. Since all round bubbles are simple, they may have at most one singular branched point. However, there exists no conformal parametrization of the Euclidean sphere with one single branch point.  Accordingly, none of the bubbles $\omega^i$ may have branch points, and thus are all immersions. \\

Next, let $x_k \in \D_{\frac{1}{2} }$ be such that:
\begin{equation}
\label{inx0}
\left| \Ar_k e^{-\lambda_k} \right|(x_k) \ge C(k) \left\| \Ar e^{-\lambda_k} \right\|_{L^2 \left( \D \right)}.
\end{equation}
There exists $x_0 \in \D_{\frac{3}{4} } $  such that $x_k \rightarrow x_0$. Necessarily, $x_0$ is a concentration point (one of the aforementioned points $a^i$). We choose $\rho>0$ such that $B(x_0,\rho)$ does not contain any other concentration point (since those are isolated), and moreover
$$ \xi_k^i \circ \bp_k( \rho_k^i x  +a_k^i) \rightarrow \omega \quad C^\infty_{\mathrm{loc}} (\R^2 ),$$
where $\omega$ parametrizes a round sphere. Consider $p:=\omega(\infty)$ and 
 $\iota_p$ the inversion of $\R^3$ centered on $p$. Put $\bpsi_k := \iota_p \circ \xi_k^i \circ  \bp_k $. If the energy $\Vert \nabla \vec{n}_{k}\Vert_2$ were to concentrate, we would be able to blow-up a round sphere, but since a plane ($\iota_p \circ \omega$) develops at scale $\rho^i_k$, using again the argument of the proof of theorem 0.2 of Laurain and Rivière \cite{MR3843372} and recalled above, it is then possible to generate a non-umbilic bubble between the sphere and the plane (in the same manner as one proves the simplicity of the bubbles), which yields a contradiction.\\ 

We may now apply theorem \ref{mainth1} to $\bp_k$ on a finite cover of $B(x_0, \rho)$, thereby obtaining by conformal invariance the estimate
\begin{equation*}
    \left\Vert \Ar_k e^{-\lambda_k} \right\Vert_{L^\infty(B\left(x_0,\frac{\rho}{2}\right))} \leq C \left\| \Ar e^{-\lambda_k} \right\|_{L^2 \left( \D \right)}.
\end{equation*}
This contradicts \eqref{HN1} and concludes the proof of Theorem \ref{mainth2}.
\end{proof}

\begin{remark}
In the proof of theorem \ref{mainth2}, we do not exclude Euclidean spheres as bubbles. We merely show they appear through a more regular concentration phenomenon, one not affecting the tracefree curvature.  A parallel reasoning should be drawn with \cite{bibnmheps}, where the $\varepsilon$-regularity for $H$  yields an improved regularity for minimal bubbling, because of the high impact a control on the mean curvature has on the regularity of the immersion. In the present case, while control on the tracefree curvature does not immediately yield control on the immersions, it sufficiently restricts the appearance of bubbles. 
\end{remark}

\subsection{An Umbilical Willmore Bubble}
In order to clarify the proof of theorem \ref{mainth2}, it is instructive to consider a concrete example of an umbilic Willmore bubble. In the first section of \cite{bibnm3}, an example of Willmore bubbling with a minimal bubble is given: there exists a sequence of Willmore immersions $\bp_\mu$  of the sphere converging smoothly to an inverted L\'opez surface, away from a single concentration point. We remind the reader that a L\'opez surface is a minimal sphere with one branched end of multiplicity $3$ and one immersed end.  The inverted L\'opez surface is thus a Willmore branched surface (see figure \ref{figurelopez}), with a point of density $4$, decomposed into a branch point of multiplicity $3$ and a regular point. \\

\begin{figure}
\includegraphics{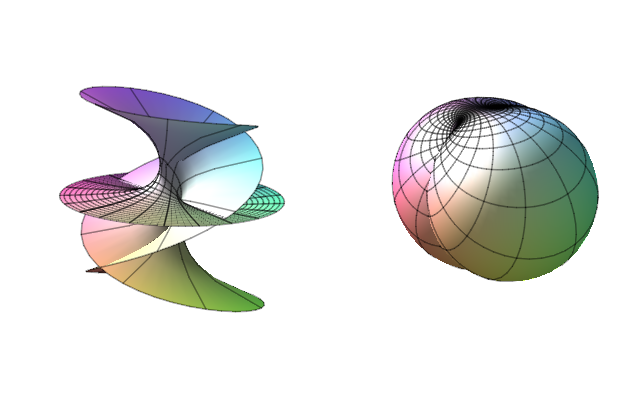}
\caption{A L\'opez surface, and one of its inversions}
\label{figurelopez}
\end{figure}

The appearance of a branch point is symptomatic of bubbling phenomena, and a blow-up analysis shows that the sequence $\frac{\bp_\mu (\mu^3.)}{\mu^9}$ converges smoothly to an Enneper surface (a minimal sphere with one branched end of multiplicity $3$, see figure \ref{figureenneper}).\\

\begin{figure}
\includegraphics{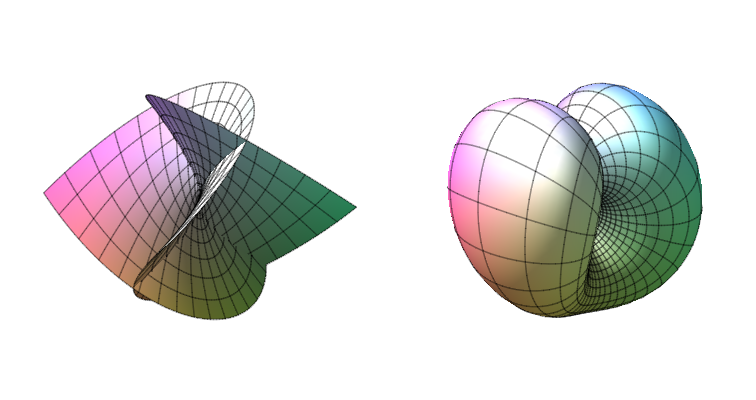}
\caption{An Enneper surface, and one of its inversions}
\label{figureenneper}
\end{figure}

Given the complexity of situations involving multiple points (branch ends and branch points) we adopt a schematic representation to illustrate the bubbling configurations. The L\'opez surface (and its inverse) will be represented according to figure \ref{schemalopez} while the Enneper (and its inverse) will be represented by figure \ref{schemaenneper}. The bubbling configuration of $\bp_\mu$ is schematically depicted on figure \ref{schemaPhimu}.\\
\begin{figure}
\begin{center}
\begin{tikzpicture}[scale=1]
\begin{scope}[xshift=0cm, yshift=0cm]
  \draw (0,0) arc(100:440:1);
\draw[red] (0.05,-0.1)--(0.35,0.1);
\draw[red] (0.05,0.1)--(0.35,-0.1);
\draw[red] (0.2,-0.15)--(0.2,0.15);
\draw[blue] (0.05,0)--(0.325,0);
\end{scope}
\begin{scope}[scale = 0.5, xshift=-6cm, yshift=-2cm]
  \draw (-1,-1)--(-1,1);
\draw (1,-1)--(1,1);
\draw[red]  (1,1) arc(180:90:1);
\draw[red] (-1,1) arc(0:90:1);
\draw[red]  (1,1) arc(180:90:0.5);
\draw[red] (-1,1) arc(0:90:0.5);
\draw[red]  (1,1) arc(180:90:0.75);
\draw[red] (-1,1) arc(0:90:0.75);
\draw[blue] (-1,-1) arc(0:-90:1);
\draw[blue] (1,-1) arc(180:270:1);
\end{scope}
\end{tikzpicture}
\end{center}
\caption{Schematic representation of the L\'opez surface and its inverse}
\label{schemalopez}
\end{figure}
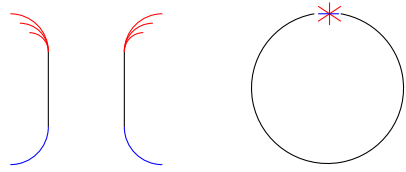

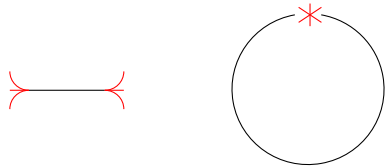
\begin{figure}
\begin{center}
\begin{tikzpicture}[scale=1]
\begin{scope}[xshift=0cm, yshift=0cm]
  \draw (0,0) arc(100:440:1);
\draw[red] (0.05,-0.1)--(0.35,0.1);
\draw[red] (0.05,0.1)--(0.35,-0.1);
\draw[red] (0.2,-0.15)--(0.2,0.15);
\end{scope}
\begin{scope}[scale = 0.5, xshift=-6cm, yshift=-2cm]
  \draw (-1,0)--(1,0);
\draw[red]  (-1,0) arc(-90:-180:0.5);
\draw[red]  (-1,0) arc(90:180:0.5);
\draw[red]  (-1,0)--(-1.5,0);
\draw[red]  (1,0) arc(-90:0:0.5);
\draw[red]  (1,0) arc(90:0:0.5);
\draw[red]  (1,0)--(1.5,0);
\end{scope}
\end{tikzpicture}
\end{center}
\caption{Schematic representation  of the Enneper surface and its inverse}
\label{schemaenneper}
\end{figure}

\begin{figure}
\begin{center}
\begin{tikzpicture}[scale=1]
\begin{scope}[xshift=0cm, yshift=0cm]
  \draw (0,0) arc(100:440:1);
\draw[red] (0.05,-0.1)--(0.35,0.1);
\draw[red] (0.05,0.1)--(0.35,-0.1);
\draw[red] (0.2,-0.15)--(0.2,0.15);
\draw[blue] (0.05,0)--(0.325,0);
\end{scope}
\begin{scope}[scale = 0.15, xshift=1.25cm, yshift=2cm]
 \draw (-1,0)--(1,0);
\draw[red]  (-1,0) arc(-90:-180:0.5);
\draw[red]  (-1,0) arc(90:180:0.5);
\draw[red]  (-1,0)--(-1.5,0);
\draw[red]  (1,0) arc(-90:0:0.5);
\draw[red]  (1,0) arc(90:0:0.5);
\draw[red]  (1,0)--(1.5,0);
\end{scope}
\end{tikzpicture}
\end{center}
\caption{Minimal bubbling: an Enneper bubble desingularizes the branch point of an inverted L\'opez surface}
\label{schemaPhimu}
\end{figure}
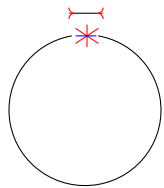

In what follows, we exploit the topological symmetry of this bubbling configuration: \emph{both the limit surface and the bubble are topological spheres}. We can then reverse the situation and consider the compactified bubble as the limit surface. Let us then consider $p \in \R^3$ and $\eta>0$ such that $d \left(p, \frac{\bp_\mu(\mu^3.)}{\mu^9} \right)\ge \eta >0$, and define:$$\bpsi_\mu (z):= \iota_p \circ \left( \frac{\bp_\mu(\frac{\mu^3}{z})}{\mu^9} \right).$$
Necessarily, $\bp_\mu$ converges to an inverted Enneper surface whose branch point at $0$ must be desingularized by a non-compact bubble. Let
$$\tilde \bpsi_\mu := \frac{\bpsi_\mu (\mu^3\,\cdot)}{\mu^9} = \frac{ \iota_p \circ \left( \frac{\bp_\mu(\frac{\mu^3}{\mu^3 z})}{\mu^9} \right) }{\mu^9} = \iota_{\mu^9p} \circ \bp_\mu \left( \frac{1}{z} \right).$$
Given the asymptotic behavior of $\bp_\mu$, the map $\tilde \bpsi_\mu$ may only converge to a L\'opez surface (whose branched end is this time at $\infty$ while its simple end is at $0$). Since a L\'opez minimal surface has two ends, i.e. two singular points, $\tilde \bpsi_\mu$ still has a concentration point desingularizing the simple end. However, given the conformal invariance of $\big| \Ar e^{-\lambda} \big|$, all the tracefree total curvature is accounted for within the inverted Enneper and the L\'opez surfaces. The only remaining bubble must then be totally umbilic, that is, a Euclidean sphere. This bubbling configuration is schematically represented on figure \ref{bulleombilique}.

\begin{figure}
\begin{center}
\begin{tikzpicture}[scale=1]
\begin{scope}[xshift=0cm, yshift=0cm]
  \draw (0,0) arc(100:440:1);
\draw[red] (0.05,-0.1)--(0.35,0.1);
\draw[red] (0.05,0.1)--(0.35,-0.1);
\draw[red] (0.2,-0.15)--(0.2,0.15);
\end{scope}
\begin{scope}[scale = 0.15, xshift=1.25cm, yshift=2.5cm,rotate =180]
  \draw (-1,-1)--(-1,1);
\draw (1,-1)--(1,1);
\draw[red]  (1,1) arc(180:90:1);
\draw[red] (-1,1) arc(0:90:1);
\draw[red]  (1,1) arc(180:90:0.5);
\draw[red] (-1,1) arc(0:90:0.5);
\draw[red]  (1,1) arc(180:90:0.75);
\draw[red] (-1,1) arc(0:90:0.75);
\draw[blue] (-1,-1) arc(0:-90:1);
\draw[blue] (1,-1) arc(180:270:1);
\end{scope}
\begin{scope}[scale=0.5,xshift=0.5cm, yshift=1.5cm, rotate =180]
\draw[blue] (0,0) arc(100:440:1);
\end{scope}
\end{tikzpicture}
\end{center}
\caption{Non-simple bubbling: on an inverted Enneper, one glues first a L\'opez surface, and then a euclidean sphere}
\label{bulleombilique}
\end{figure}
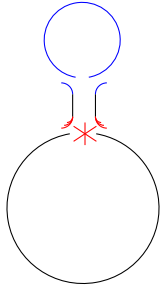

Inspecting $  \frac{\left.\tilde \bpsi_\mu\right|_{\D}}{\mathrm{diam} ( \tilde \bpsi_\mu (\D))}$, we observe the situation described in theorem \ref{mainth2}: a Willmore disk becoming a non compact surface up to a proper rescaling, as a Euclidean bubble concentrates at the origin. However, this bubble arises after an inversion at a regular point of another bubbling configuration (the blue regular sheet of figure \ref{schemaPhimu} becomes the blue end and the blue bubble of \ref{bulleombilique}). This configuration is thus regular and does satisfy \eqref{Mest3}.  One key consequence of the proof of theorem \ref{mainth2} is then that all umbilic bubbles are ``artificial" ones arising from the inversion of a regular point, and are thus smoother than expected.

\section{Applications}
\subsection{Proof of Theorem \ref{theoconfGauss}: Lorentzian \texorpdfstring{$\varepsilon$}{TEXT}-regularity}
\begin{proof}
The statement is merely a reformulation of \ref{mainth2} with $\langle \nabla Y, \nabla Y \rangle_{4,1} = \big| \Ar e^{-\lambda} \big|^2$ (see proposition \ref{230520200838}).
\end{proof}

\subsection{Proof of Theorem \ref{gap}: Intrinsic Gap}
\begin{proof}
We first show that $\Sigma$ is topologically either a plane or a sphere. By a classical result of Huber \cite{Huber}, see also \cite{White}, we know that $\Sigma$ is conformally equivalent to a compact Riemann surface $\hat{\Sigma}$ with possibly a finite set of points $\{p_1,\dots,p_N\}$ removed. So there exists a conformal parametrization $\bp : \hat{\Sigma}\setminus \{p_i\}_i \rightarrow \Sigma$. Since $\Sigma$ is complete, each $p_i$ corresponds to some end of $\Sigma$. Hence, thanks to the generalized Gauss-Bonnet formula (see theorem 10 of \cite{Ng}), we have   
$$\int_\Sigma K \, d\sigma = 4\pi (1-g(\hat{\Sigma})) -2 \pi \sum_{i=1}^n (m_i+1),$$  
where $g(\hat{\Sigma})$ is the genus of $\hat{\Sigma}$ and $m_i$ is the multiplicity of the end $p_i$. Moreover we have on $\Sigma$ the identity
$$\vert A\vert_g^2 =2 \vert\mathring{A}\vert_g^2 + 2K,$$
so that 
\begin{equation}
    \label{GB}
    \frac{1}{2} \int_{\Sigma} \vert A \vert_g^2 \, d\text{vol}_g \leq \int_\Sigma \vert \mathring{A} \vert_g^2 \, d\text{vol}_g +  4\pi (1-g(\hat{\Sigma})) -4\pi N .
\end{equation}
Then for $\varepsilon_0$ small enough, $\hat{\Sigma}$ is a topological sphere with at most one end or a torus with no end. But the latter case is excluded since the right-hand side of \eqref{GB} is bigger than $4\pi$ owing to a classical estimate by T. Willmore (theorem 7.2.2 \cite{bibwill}).\\

In conclusion, there exists a conformal Willmore immersion $\bp : \R^2 \rightarrow \R^2$ whose image is $\Sigma$ (up to the removal of a point in the compact case). Applying corollary \ref{coro} around any point $p\in \R^2$ and letting $\rho \rightarrow +\infty$ yields $\vert \mathring{A}\vert(p) =0$, which implies the announced statement. 
\end{proof}

\section*{Appendix}

\subsection*{Variational bubbles are conformally minimal}
In theorem H of \cite{michelatclassifi}, the authors obtain as a byproduct that any branched sphere that appears as a bubble must be conformally minimal. The proof of theorem H in \cite{michelatclassifi} is quite involved due to the general assumptions used by the authors. For the sake of completeness of the present paper, we give an elementary argument to obtain the same result.
\begin{lem}
    \label{mini}
    Let $\bp_k :\D \rightarrow \R^3$ be a sequence of conformal Willmore immersions, $a_k\in \D$ a sequence of points converging to some $a_\infty$, $ \mu_k$ converging to $0$, and $T_k$ a sequence of conformal transformations of $\R^3$ . Suppose
    $$\tilde{\bp}_k:=T_k\circ \bp_k(a_k+\mu_k \, \cdot ) \rightarrow \omega \hbox{ on } \R^2\setminus{S}$$
    where $S$ is the finite set and $\omega $ a branched Willmore sphere. Then $\omega$ is conformally minimal.   
\end{lem}
\begin{proof}
Following Bryant's work, in order to prove that a sphere is conformally minimal, it suffices to prove that its associated quartic form $Q_\omega$ vanishes (see theorem E in \cite{bibdualitytheorem}).
    Let $Q_{\tilde{\bp}_k}$ be the quartic form associated with $\bp_k$. Per theorem B in \cite{bibdualitytheorem}, $Q_{\tilde{\bp}_k}$ is holomorphic. Let $p\in R^2$ be a branch point. Our strong convergence hypothesis away from branch points guarantees that the quartic form $Q_\omega$ is holomorphic around $p$, because $Q_{\tilde{\bp}_k}$ is bounded.\, owing to the maximum principle. Hence $Q_\omega$ may have at most one pole at infinity. Letting $\tilde{Q}= Q_{\omega}\left( \frac{1}{z} \right)$, by theorem 3.1 of \cite{LN}, the order of the pole of $\tilde{Q}$ at $0$ is at most $2$. One easily checks that  $\tilde{Q}(z)=O\left(\frac{1}{\vert z\vert^8}\right)$, so that $Q_\omega\equiv 0$ by Liouville's theorem, thereby concluding the proof.
\end{proof}
\subsection*{Convenient Reformulation of the Gauss-Codazzi equation}
Recall that $\nabla \n = -e^{-2\lambda} A \nabla \Phi$, that is 
$$ \begin{aligned}
-\n_x &= \left( H + \left( \frac{l-n}{2} \right)e^{-2\lambda} \right) \Phi_x + m e^{-2\lambda} \Phi_y \\
\text{and}\\
-\n_y &= m e^{-2\lambda} \Phi_x + \left( H - \left( \frac{l-n}{2} \right)e^{-2\lambda} \right) \Phi_y.
\end{aligned}$$

Differentiating yields

$$\begin{aligned}
- \n_{xy} &= \left( H_y+ \left( \frac{l-n}{2} \right)_y e^{-2\lambda} + \lambda_y   H - \lambda_y \left( \frac{l-n}{2} \right)e^{-2\lambda}  - \lambda_x m e^{-2\lambda} \right) \Phi_x  \\& + \left( m_y e^{-2\lambda}  - \lambda_y m e^{-2 \lambda} + \lambda_x  \left( H + \left( \frac{l-n}{2} \right) \right) \right) \Phi_y + \left(\dots \right) \n \\
\text{and}\\
-\n_{yx} &= \left( m_x  e^{-2\lambda} - \lambda_x m e^{-2\lambda}  + \lambda_y \left( H - \left( \frac{l-n}{2} \right)e^{-2\lambda} \right) \right) \Phi_x  \\&+ \left( H_x - \left( \frac{l-n}{2} \right)_x e^{-2\lambda} + \lambda_x H  + \lambda_x \left( \frac{l-n}{2} \right)  - \lambda_y m e^{-2\lambda}  \right) +  \left( \dots \right) \n.
\end{aligned}$$

Identifying $\n_{xy}$ and $\n_{yx}$, one finds:
\begin{equation}
\label{Gausscodazziformereelle}
\begin{aligned}
e^{2\lambda} H_x &=  \left( \frac{l-n}{2} \right)_x + m_y \\
\text{and}\\
e^{2\lambda} H_y &= -  \left( \frac{l-n}{2} \right)_y + m_x.
\end{aligned}
\end{equation}

\subsection*{Brief Proof of Theorem \ref{theoZhu}}
In order to contrast \ref{theoconfGauss} from \ref{theoZhu}, we sketch a proof of theorem  \ref{theoZhu}.
\begin{proof}
A weakly harmonic application $u$ satisfies: $$\Delta u + \langle \nabla u , \nabla u \rangle_{4,1} u  =0.$$ Equivalently, the latter may be recast in the conservative form
$$\mathrm{div}( \nabla u u^T - u \nabla u^T ) = 0.$$ Accordingly, on $\D$,  there exists a matrix $B$ such that
$\nabla^\perp B = \nabla u u^T - u \nabla u^T$. As a consequence, $B$ satisfies $\nabla B = u \nabla^\perp u^T - \nabla^\perp u u^T$ which implies
\begin{equation}
\label{eqmatB}
\Delta B = 2 \nabla u \nabla^\perp u^T.
\end{equation}
On the other hand, if we denote by $\epsilon$ the signature matrix of $\R^{4,1}$, then for any two vectors $a,b$, we have $\langle a, b \rangle_{4,1}= a^T \epsilon b$. 
Then
$$\nabla^\perp B \epsilon \nabla u = \nabla u u^T \epsilon \nabla u  - u \nabla u^T \epsilon \nabla u = \langle u, \nabla u \rangle_{4,1} \nabla u - \langle \nabla u , \nabla u \rangle_{4,1} u = - \langle \nabla u , \nabla u \rangle_{4,1} u  = \Delta u,$$
since, given that $u\in \s^{4,1}$,  we have $\langle u, \nabla u \rangle_{4,1}=0$. Combining this to \eqref{eqmatB} yields the system:
\begin{equation*}
\begin{aligned}
\Delta u &= \nabla^\perp B \nabla (\epsilon  u ) \\
\Delta B& =  2 \nabla u \nabla^\perp u^T.
\end{aligned}
\end{equation*}
Provided that $\left\| \nabla  u \right\|_{L^2 } \le \varepsilon_0$, the statement \eqref{epsregZhu} now follows from classical integration by compensation techniques, such as those presented in \cite{bibharmmaps}.
\end{proof}
\addcontentsline{toc}{section}{Bibliography}
\bibliographystyle{plain}
\bibliography{bibliography}

\end{document}